\documentclass[12pt]{amsart}
\usepackage{amsthm,amssymb,amsmath,hyperref, cite}
\usepackage{color, mdframed, tikz}
\usepackage{graphicx,enumerate}
\pdfoutput=1
\usepackage[margin=1in]{geometry}
\allowdisplaybreaks
\usetikzlibrary{decorations.pathreplacing,calc}

\makeatletter
\def\l@section{\@tocline{1}{12pt plus2pt}{0pt}{}{\bfseries}}
\def\l@subsection{\@tocline{2}{0pt}{2pc}{2pc}{}}
\makeatother
\makeatletter
\def\subsection{\@startsection{subsection}{2}{\z@}%
	{-3.25ex\@plus -1ex \@minus -.2ex}%
	{1.5ex \@plus .2ex}%
	{\normalfont\bfseries\boldmath}}
\def\subsubsection{\@startsection{subsubsection}{3}%
	\z@{.5\linespacing\@plus.7\linespacing}{-.5em}%
	{\normalfont\bfseries\boldmath}}
\renewcommand\paragraph{\@startsection{paragraph}{4}{\z@}%
	{3.25ex \@plus1ex \@minus.2ex}%
	{-1em}%
	{\normalfont\normalsize\bfseries}}

    \makeatletter
\newcommand{\widecheck}[1]{\mathpalette\widecheck@{#1}}
\newcommand{\widecheck@}[2]{%
  \begingroup
  \setbox0=\hbox{$\m@th#1#2$}%
  \setbox2=\hbox{$\m@th#1\widehat{\vphantom{#2}\hphantom{#2}}$}%
  \mathord{%
    \ooalign{%
      \hfil
      \raisebox{\dimexpr\ht2+\ht0+.15ex\relax}{%
        \scalebox{1}[-1]{\copy2}%
      }%
      \hfil\cr
      \hfil$\,\m@th#1#2\,$\hfil\cr
    }%
  }%
  \endgroup
}
\makeatother

\definecolor{myblue}{RGB}{45,70,200}
\definecolor{lightblue}{RGB}{220,232,248}
\definecolor{mygreen}{RGB}{40,140,60}
\definecolor{lightgreen}{RGB}{222,239,222}
\definecolor{myred}{RGB}{200,50,40}
\definecolor{lightred}{RGB}{248,225,225}

\allowdisplaybreaks

\theoremstyle{plain}
\newtheorem{thm}{Theorem}[section]

\newtheorem{lem}[thm]{Lemma}

\theoremstyle{definition}
\newtheorem{defn}[thm]{Definition}

\theoremstyle{remark}
\newtheorem{rem}[thm]{Remark}

\newtheorem{exa}[thm]{Example}
\theoremstyle{plain}

\numberwithin{equation}{section}

\theoremstyle{plain} 
\newcommand{\thistheoremname}{}
\newtheorem{genericthm}[thm]{\thistheoremname}

  \newtheorem*{genericthm*}{\thistheoremname}
\newenvironment{namedthm*}[1]
  {\renewcommand{\thistheoremname}{#1}%
   \begin{genericthm*}}
  {\end{genericthm*}}

\newcommand{\norm}[1]{\left\Vert#1\right\Vert}

\newcommand{\eps}{\varepsilon}

\newcommand{\ze}{\zeta}

\newcommand{\D}{{\mathbb D}}

\newcommand{\N}{{\mathbb N}}

\newcommand{\calB}{{\mathcal B}}

\newcommand{\calD}{{\mathcal D}}

\newcommand{\calI}{{\mathcal I}}

\newcommand{\calM}{{\mathcal M }}
\newcommand{\calN}{{\mathcal N}}

\newcommand{\calQ}{{\mathcal Q}}

\newcommand{\calT}{{\mathcal T}}

\newcommand{\supp}{{\textrm{supp}}}

\makeatletter
\newcommand{\vast}{\bBigg@{4}}
\newcommand{\Vast}{\bBigg@{5}}
\makeatother

\newtheorem{opques}[thm]{Open Question}

\usepackage{scalerel}[2014/03/10]
\usepackage[usestackEOL]{stackengine}
\def\intavg{\,\ThisStyle{\ensurestackMath{%
    \stackinset{c}{0\LMpt}{c}{0\LMpt}{\SavedStyle-}{\SavedStyle\phantom{\int}}}%
    \setbox0=\hbox{$\SavedStyle\int\,$}\kern-\wd0}\int}

\def\udot#1{\ifmmode\oalign{$#1$\crcr\hidewidth.\hidewidth
    }\else\oalign{#1\crcr\hidewidth.\hidewidth}\fi}

\def\T{\mathbb{T}}

\def\beq{\begin{equation}}
\def\eeq{\end{equation}}
\makeatletter
\newcommand{\doublewidetilde}[1]{{%
  \mathpalette\double@widetilde{#1}%
}}
\newcommand{\double@widetilde}[2]{%
  \sbox\z@{$\m@th#1\widetilde{#2}$}%
  \ht\z@=.9\ht\z@
  \widetilde{\box\z@}%
}
\makeatother

\def\one{\mbox{1\hspace{-4.25pt}\fontsize{12}{14.4}\selectfont\textrm{1}}}

\usepackage{tikz}

\makeatletter
\def\@makefnmark{%
  \leavevmode
  \raise.9ex\hbox{\fontsize\sf@size\z@\normalfont\tiny\@thefnmark}}
\makeatother

\begin{document}
	
\title[]{A non-testing characterization of bounded and compact composition operators on $\calQ_p$ spaces}

\author{Bingyang Hu}
\address{(Bingyang Hu) Department of Mathematics and Statistics\\
        Auburn University\\
        Auburn, Alabama, U.S.A, 36849}
\email{bzh0108@auburn.edu}

\author{Xiaojing Zhou}
\address{(Xiaojing Zhou) Department of Mathematics and Statistics\\
         Auburn University\\
         Auburn, Alabama, U.S.A, 36849}
\email{xiz0003@auburn.edu}

\begin{abstract}

In this paper, we give symbol-only, non-testing characterizations of bounded and
compact composition operators on $\calQ_p$, $0<p\le 1$, via a novel dyadic trace formulation over Carleson tents for generalized $p$-Nevanlinna counting
functions. Our results resolve an open question raised in Xiao's 2001 book, as well as the diagonal case of Zhao's 2009 question, a longstanding problem in the
theory of $\calQ_p$ spaces. As an application of our main theorems, we also obtain boundedness and compactness criteria for the $\calQ_p$-Carleson measure embedding $\mathrm{id}:\calQ_p\to L^2(WdA)$ for weights $W$ in the
Littlewood--Paley class.

\end{abstract}
\date{\today}


\maketitle


\section{Introduction}
The goal of this paper is to study composition operators on the spaces $\calQ_p$, $0<p\leq 1$, from the point of view of dyadic harmonic analysis over Carleson tents.  The main point is to give a criterion which depends only on the symbol and does not involve testing over all functions in $\calQ_p$, thereby answering a question raised in Xiao's 2001 book~\cite{Xiao2001} and
later formulated by Zhao~\cite{Zhao2009} in 2009.

We begin with some definition.  Let $\D$ be the unit disc, let $dA$ denote normalized area measure, and $H(\D)$ be the collection of all holomorphic functions on $\D$ equipped with the compact--open topology.  For $a\in\D$, let
\[
        \sigma_a(z)=\frac{a-z}{1-\overline a z}
\]
be the automorphism of $\D$ exchanging $0$ and $a$.  The Green function of the unit disc is
\[
        g(z,a)=\log\frac1{|\sigma_a(z)|}.
\]
For \(0<p<\infty\), the space \(\calQ_p\) consists of all analytic functions \(f \in H(\D) \) on \(\D\) such that
\[
\norm{f}_{\calQ_p,*}^2:=\sup_{a\in\D}\int_{\D}|f'(z)|^2g(z,a)^p\,dA(z)<\infty.
\]
Here \(\norm{\cdot}_{\calQ_p,*}\) is a seminorm. We equip \(\calQ_p\) with the norm
\[
\norm{f}_{\calQ_p}:=|f(0)|+\norm{f}_{\calQ_p,*}.
\]  
It is well known that 
\begin{enumerate}
    \item [$\bullet$] when $p=1$, $\calQ_1=BMOA$, the space of bounded mean oscillation analytic; 
    \item [$\bullet$] when $p>1$, $\calQ_p=\calB$, the Bloch space;
    \item [$\bullet$] when $0<p_1<p_2<1$,  $\mathfrak D \subsetneq \calQ_{p_1}\subsetneq \calQ_{p_2}\subsetneq BMOA$, where $\mathfrak D$ is the Dirichlet space. 
\end{enumerate}
To this end, we recall the following Carleson measure characterization of \(\calQ_p\). For \(0<p<\infty\), one has
\begin{equation} \label{20260606eq01}
\norm{f}_{\calQ_p,*}^2
\simeq
\sup_{I\subseteq\T}
\frac{1}{|I|^p}
\int_{Q_I}|f'(z)|^2(1-|z|^2)^p\,dA(z),
\end{equation} 
see, for instance, \cite[Theorem~1.1.3]{Xiao2006}. Here \(\T\) denotes the unit circle, \(I\subseteq\T\) ranges over all arcs, and \(Q_I\) is the Carleson tent associated with \(I\); see, \eqref{20260607eqO1}. Equivalently, \eqref{20260606eq01} says that \(f\in\calQ_p\) if and only if the measure
\[
        d\mu_f(z):=|f'(z)|^2(1-|z|^2)^p\,dA(z)
\]
is a \(p\)-Carleson measure. Motivated by \eqref{20260606eq01}, we define the tent space \(\calT_p\) as follows. For \(0<p<\infty\), \(\calT_p\) consists of all measurable functions $F$ on $\D$ satisfying
\[
\norm{F}_{\calT_p}
:=
\left(
\sup_{I\subseteq\T}
\frac{1}{|I|^p}
\int_{Q_I}|F(z)|^2(1-|z|^2)^p\,dA(z)
\right)^{1/2}
<\infty.
\]
Therefore, $f \in \calQ_p$ if and only if $f' \in \calT_p$, and further one has 
\begin{equation} \label{20260606D1}
\left\|f\right\|_{\calQ_p, *} \simeq \left\|f'\right\|_{\calT_p}
\end{equation} 

\vspace{0.1cm}

Given an analytic self-map $\varphi$ of $\D$, the composition operator $C_\varphi$ is defined by
\[
        C_\varphi f=f\circ\varphi, \qquad f \in H(\D). 
\]
A longstanding open problem in the theory of $\calQ_p$ spaces is to characterize the boundedness and compactness of $C_\varphi$ acting on $\calQ_p$ with $0<p <1$. Previous criteria for this problem are formulated in terms of strong testing conditions involving both the symbol $\varphi$ and functions $f\in \calQ_p$; see, for example, \cite{Lou2001,WirthsXiao2002}. Such characterizations are therefore not conditions purely in terms of the symbol $\varphi$, which are usually considered not very satisfactory. This symbol-only
diagonal problem was already highlighted by Xiao in \cite[p.~22, Note~2.2]{Xiao2001}, and Zhao~\cite{Zhao2009} later formulated the following broader question.

\begin{opques}
What are the conditions in terms of $\varphi$ so that $C_\varphi$ is bounded
or compact on $\calQ_p$, or between different $\calQ_p$ spaces, for
$0<p<1$?
\end{opques}

The main \emph{goal} of this paper is to give a complete answer to the diagonal
case of the above open question, namely to obtain symbol-only, non-testing
characterizations for the boundedness and compactness of
\[
        C_\varphi:\calQ_p\longrightarrow\calQ_p,
        \qquad 0<p<1.
\]
We expect that the methods developed here can also be adapted to the
off-diagonal problem, namely to composition operators between different
\(\calQ_p\) spaces, but we leave this direction for future work.

\medskip 

Before turning to the formulation of our main results, let us briefly recall the endpoint case \(p=1\), for which \(\calQ_1=BMOA\). In this case the theory of composition operators is well-understood. 

\begin{enumerate}
    \item \textbf{Boundedness.} If \(\varphi\) is an analytic self-map of \(\D\), then \(C_\varphi\) is automatically bounded on BMOA. This follows from the Littlewood's subordination principle; see, for instance, \cite{Stein1970}.

    \vspace{0.1cm}

    \item \textbf{Compactness.} The compactness problem is considerably subtler. Bourdon, Cima and Matheson \cite{BCM1999} first characterized compactness on BMOA by a little-oh Carleson measure condition, uniform over the unit ball of BMOA. Smith \cite{Smith1999} subsequently obtained a function-theoretic characterization depending only on the symbol. More precisely, if
    \[
        E(\varphi,t):=\{e^{i\theta}: |\varphi(e^{i\theta})|>t\},
        \qquad 0<t<1,
    \]
    then \(C_\varphi\) is compact on BMOA if and only if
    \[
        \lim_{|\varphi(a)|\to 1}
        \sup_{0<|w|<1}
        |w|^2
        N_{\sigma_{\varphi(a)} \circ \varphi \circ \sigma_a}(w)
        =0
    \]
    and, for every \(0<R<1\),
    \[
        \lim_{t\to1^-}
        \sup_{|\varphi(a)|\leq R}
        m\bigl(\sigma_a(E(\varphi,t))\bigr)
        =0.
    \]
   
\end{enumerate}

This line of research was further developed in several subsequent works. Wulan \cite{Wulan2007} gave simpler compactness conditions on BMOA and VMOA. Wulan, Zheng and Zhu \cite{WulanZhengZhu2009} further obtained compactness criteria for \(C_\varphi\) on BMOA and the Bloch space in terms of the norms of the powers \(\varphi^n\) in the corresponding spaces. Finally, Laitila, Nieminen, Saksman and Tylli \cite{LNST2013} simplified Smith's criterion by showing that the Nevanlinna counting function condition alone characterizes compactness on BMOA; they also proved that weak compactness of \(C_\varphi\) on BMOA is equivalent to compactness. In a related direction, we also refer the reader to several beautiful extensions on the Euclidean counterpart of the above theme, where the quasiconformality of the symbol plays an essential role; see \cite{EJPX2000,KoskelaXiaoZhangZhou2017,XiaoZhou2019}.

\vspace{0.1cm}

This endpoint theory provides an important guide for the present paper. However, the proof of the \(p=1\) theory relies in an essential way on the linear logarithmic structure of the classical Nevanlinna counting function, a feature which is no longer available in the range \(0<p<1\).

Let us recall the main obstruction. The key ingredient in Smith's approach is Littlewood's inequality for the classical Nevanlinna counting function. Recall that, for an analytic self-map \(\varphi\) of \(\D\), the Nevanlinna counting function is defined by
\begin{equation} \label{20260606eqR1}
       N_\varphi(w):=\sum_{\varphi(z)=w}\log\frac{1}{|z|},
        \qquad w\in \D\setminus\{\varphi(0)\},
\end{equation} 
where the sum is counted with multiplicity. Littlewood's inequality asserts that
\[
        N_\varphi(w)\leq N_{\sigma_{\varphi(0)}}(w)
        =\log\frac{1}{|\sigma_{\varphi(0)}(w)|}.
\]
This estimate is the mechanism behind the boundedness of \(C_\varphi\) on BMOA, and its M\"obius invariant little-oh form is the starting point of Smith's compactness criterion.

The proof of Littlewood's inequality is intrinsically tied to the case \(p=1\): the logarithmic kernel is the Green kernel, Jensen's formula gives the corresponding linear zero-counting identity, and the partial Nevanlinna counting functions have the required subharmonicity. For \(0<p<1\), the natural \(p\)-Nevanlinna counting function
\[
        N_{\varphi,p}(w):=\sum_{\varphi(z)=w}
        \left(\log\frac{1}{|z|}\right)^p
\]
does \emph{not}, in general, possess the corresponding subharmonicity structure. Consequently, the Littlewood-type pointwise majorization underlying Smith's method is no longer available. This is one of the main reasons why Smith's approach has not been extendable to the range \(0<p<1\) since Xiao and Zhao posed the question.

\medskip 

The obstruction above suggests that, in the range \(0<p<1\), one should \emph{not} expect a useful characterization in terms of the pointwise behavior of the \(p\)-Nevanlinna counting function. Such pointwise conditions, in the spirit of Littlewood's inequality \eqref{20260606eqR1}, are generally too strong. Instead, we take a different approach: we reformulate the problem using ideas from \emph{dyadic harmonic analysis over Carleson tents}.

\vspace{0.1cm}

In this paper, we shall focus on the most interesting range \(0<p\leq 1\). Before stating our main results, we introduce the notation and definitions needed for their formulation.

For an analytic self-map $\varphi$, $0<p \le 1$, and $a\in\D$, define the generalized $p$-Nevanlinna counting function associated to $a$ by 
$$        
\calN_{p,\varphi,a}(\ze)
:=\sum_{\varphi(z)=\ze}g^p(z,a)=\sum_{\varphi(z)=\ze} \left(\log \frac{1}{|\sigma_a(z)|} \right)^p,
        \qquad \ze\in\D. 
$$
The main reason for introducing $\mathcal N_{p,\varphi,a}$ is the following
change-of-variables formula:
\begin{equation} \label{20260606E1}
        \int_E \calN_{p,\varphi,a}(\ze) dA(\ze)
        =\int_{\varphi^{-1}(E)}g(z,a)^p|\varphi'(z)|^2 dA(z)
\end{equation}
for every Borel set $E\subseteq\D$.  This defines a natural counting measure associated with $C_\varphi$ and the $\calQ_p$ seminorm.

We next introduce a dyadic capacity condition. Let $\mathcal D$ be a dyadic
system on $\mathbb T$. For $I\in\mathcal D$, let $Q_I$ denote the associated
Carleson tent,
\begin{equation} \label{20260607eqO1}
        Q_I:=\left\{re^{it}: e^{it}\in I,\ 1-|I|\le r<1\right\}.
\end{equation} 
We also denote by
\[
        Q_I^{\mathrm{up}}:=
        \left\{re^{it}: e^{it}\in I,\ 1-|I|\le r<1-\frac{|I|}{2}\right\}
\]
the upper half tent of $Q_I$. For a nonnegative sequence
$h=\{h(I)\}_{I\in\mathcal D}$ and $0<p\leq 1$, define the
\emph{dyadic $p$--capacity gauge} of $h$ by
\begin{equation}\label{20260606eq04}
\norm{h}_{\mathcal B_p(\mathcal D)}
:=
\inf\left\{
        \sum_{J\in\mathcal D}\lambda_J |J|^p:
        \lambda_J\ge0,\quad
        \sum_{J\supseteq I}\lambda_J\ge h(I)\ \text{for all }I\in\mathcal D
\right\}.
\end{equation}

This gauge is motivated by the dyadic version of Hausdorff content. Indeed, the coefficients \(\{\lambda_J\}\) may be viewed as discrete covering weights on the dyadic tree, while the cost \(\sum_J\lambda_J|J|^p\) is the corresponding \(p\)-dimensional dyadic content; for background on Hausdorff measures, contents, and capacities, see Mattila~\cite{Mattila1995}. Related capacity and tree-based ideas have also played an important role in the theory of the Dirichlet space and related function spaces, from Stegenga's multiplier theorem \cite{Stegenga1980} and the work of Kerman and Sawyer \cite{KermanSawyer1988} on Carleson measures for Dirichlet-type spaces, to the elegant Bergman tree analysis developed by Arcozzi, Rochberg, Sawyer, Wick, and their coauthors; see, for instance, \cite{ArcozziRochbergSawyerWick2011,ArcozziRochbergSawyerWick2019}.

\vspace{0.1cm}

To this end, for any $a\in\D$, and $I \subseteq \T$, set
\begin{equation} \label{20260606eqF1}
        h_a^{\varphi}(I)
        :=\frac{1}{|I|^{p+2}}\int_{Q_I^{\textrm{up}}}\calN_{p,\varphi,a}(\ze) dA(\ze).
\end{equation} 
Our first main theorem is the following.

\begin{thm}[Boundedness]\label{20260606thm01}
Let $0<p\le1$ and let $\varphi$ be an analytic self-map of $\D$.  Then $C_\varphi:\calQ_p\to \calQ_p$ is bounded if and only if there exists a dyadic system $\calD$ on $\T$ such that
\begin{equation}\label{20260606eq06}
         \sup_{a\in\D}
        \norm{h_a^{\varphi}}_{\calB_p(\calD)}<\infty.
\end{equation}
Moreover,
\begin{equation} \label{20260606eq07}
        \left|\!\left|\!\left| C_\varphi \right|\!\right|\!\right|_{\calQ_p,*,\,\textnormal{op}}^2
        \simeq
        \sup_{\calD} \left[\sup_{a\in\D}
        \norm{h_a^{\varphi}}_{\calB_p(\calD)} \right]  \simeq \inf_{\calD} \left[\sup_{a\in\D}
        \norm{h_a^{\varphi}}_{\calB_p(\calD)} \right], 
\end{equation}
where the infimum and supremum are taken over all dyadic systems \(\calD\) on
\(\T\).
\end{thm}

\begin{rem}
When \(p=1\), we have \(\calQ_1=BMOA\).  In this endpoint case, every
analytic self-map \(\varphi\) of \(\D\) induces a bounded composition operator
on BMOA, by Littlewood's subordination principle.  Thus \eqref{20260606eq06} is automatic when \(p=1\).

Let us also explain how this automatic boundedness is reflected in the dyadic
trace formulation, in the sense of Theorem \ref{20260606thm03} below.  Fix \(a\in\D\), put \(b=\varphi(a)\), and set
\[
        \psi_a:=\sigma_b\circ\varphi\circ\sigma_a .
\]
Then \(\psi_a(0)=0\), and
\[
        \mathcal N_{1,\varphi,a}(\zeta)
        =
        N_{\psi_a}\bigl(\sigma_b(\zeta)\bigr).
\]
By Littlewood's inequality,
\[
        N_{\psi_a}(w)\le \log\frac1{|w|},
        \qquad w\in\D.
\]
Consequently,
\[
        \mathcal N_{1,\varphi,a}(\zeta)
        \le
        \log\frac1{|\sigma_b(\zeta)|}
        =
        g(\zeta,b).
\]
Then Littlewood--Paley identity gives, uniformly in \(b\in\D\),
\[
        \int_{\D}|F'(\zeta)|^2 g(\zeta,b)\,dA(\zeta)
        \simeq
        \|F\circ\sigma_b-F(b)\|_{H^2}^2
        \le
        \|F\|_{BMOA,*}^2 .
\]
Applying Theorem~\ref{20260606thm03} below with \(p=1\) and
\(W(\zeta)=g(\zeta,b)\), we obtain
\begin{equation} \label{20260607eq01}
        \left\|
        \left\{
        |I|^{-3}\int_{Q_I^{\mathrm{up}}}g(\zeta,b)\,dA(\zeta)
        \right\}_{I\in\calD}
        \right\|_{\calB_1(\calD)}
        \lesssim 1,
\end{equation}
with a constant independent of \(b\).  Since the dyadic $1$--capacity gauge is monotone and
\(\mathcal N_{1,\varphi,a}\le g(\cdot,\varphi(a))\), it follows that
\[
        \sup_{a\in\D}
        \|h_a^\varphi\|_{\calB_1(\calD)}
        \lesssim 1 .
\]
Thus the dyadic condition in Theorem~\ref{20260606thm01} agrees with the
classical automatic boundedness of composition operators on BMOA.  Notice
that this does not mean that the individual quantities \(h_a^\varphi(I)\) are
uniformly bounded; rather, the logarithmic singularity \eqref{20260607eq01}
is captured by the
 dyadic $1$--capacity gauge.
\end{rem}

We prove Theorem~\ref{20260606thm01} by establishing the following stronger dyadic
trace theorem, which characterizes the boundedness of the differential operator
\begin{equation} \label{20260607eqA10}
        \frac{d}{dz}: \calQ_p\longrightarrow  L^2(WdA), \qquad f \mapsto f'
\end{equation} 
for the weight \(W\).

\begin{thm} \label{20260606thm03}
Let \(0<p\le1\), and \(W\ge0\) be a locally integrable function on \(\D\). For every arc \(I\subseteq\T\), define
\begin{equation} \label{20260607A43}
        h_W(I):=\frac{1}{|I|^{p+2}}\int_{Q_I^{\textnormal{up}}}W(z)\,dA(z).
\end{equation} 
Then the following are equivalent.
\begin{enumerate}
\item[(i)] There exists a constant \(C\) such that
\begin{equation}\label{20260606eq15}
        \int_{\D}|F'(z)|^2W(z)\,dA(z)
        \le C\norm{F}_{\calQ_p,*}^2,
        \qquad F\in \calQ_p.
\end{equation}
\item[(ii)] There exists a dyadic system \(\calD\) on \(\T\) such that
\begin{equation}\label{20260606eq16}
       \norm{h_W}_{\calB_p(\calD)}<\infty,
\end{equation}
where, for each fixed \(\calD\), the sequence is understood as \(\{h_W(I)\}_{I\in\calD}\).
\end{enumerate}
Moreover, if \(C_{\mathrm{best}}\) denotes the best constant in \eqref{20260606eq15}, then
\[
        C_{\mathrm{best}}
        \simeq
        \inf_{\mathcal D}\|h_W\|_{\mathcal B_p(\mathcal D)}
        \simeq
        \sup_{\mathcal D}\|h_W\|_{\mathcal B_p(\mathcal D)},
\]
where the infimum and supremum are taken over all dyadic systems \(\mathcal D\) on \(\T\).
\end{thm}

\begin{rem}
A key observation behind Theorems \ref{20260606thm01} and \ref{20260606thm03} is that the $\calQ_p$-seminorm admits a dyadic interpretation: it can be formulated as a \emph{Carleson packing condition}; see \eqref{20260606eqB3}. This reformulation allows us to transfer the original complex-analytic problem to a problem in dyadic harmonic analysis on Carleson tents, and is the main reason for the use of dyadic techniques throughout the proof.
\end{rem}

\begin{rem} \label{20260611rem01}
Theorem~\ref{20260606thm03} is closely related to another well-known open
problem in the theory of \(\calQ_p\) spaces, namely \emph{the \(\calQ_p\)-Carleson
measure problem}\footnote{The authors thank
Ruhan Zhao for bringing this open problem to their attention.}. In one of its basic forms, this problem asks for necessary
and sufficient conditions on a positive measure \(W dA\) such that the embedding
\begin{equation}\label{20260607eqA11}
        \mathrm{id}:\calQ_p \longrightarrow L^2(W dA)
\end{equation}
is bounded, or compact. The dyadic trace formulation developed in
Theorem~\ref{20260606thm03} suggests a possible new route toward this problem.

We would like to point out, however, that the \(\calQ_p\)-Carleson measure
problem is \emph{not} a direct consequence of Theorem~\ref{20260606thm03}. Indeed,
the embedding considered in Theorem~\ref{20260606thm03} concerns the derivative
trace
\[
        F\longmapsto F',
        \qquad
        \int_{\D}|F'(z)|^2 W(z)\,dA(z),
\]
and may be viewed as an embedding of \(\calQ_p\) into a weighted Dirichlet-type
space associated with \(W\). In this setting the relevant dyadic quantity $h_W(I)$ is
local; see, \eqref{20260607A43}. By contrast, in \eqref{20260607eqA11} one has to control \(F\) itself. Since
\[
        F(z)-F(0)=\int_0^z F'(\zeta)\,d\zeta,
\]
the corresponding dyadic quantity should take into account not only the local
mass associated with \(I\), but also the contributions coming from the dyadic
ancestors of \(I\). Thus one should expect a refined dyadic
gauge, rather than a direct use of the gauge appearing in
Theorem~\ref{20260606thm03}. 

\vspace{0.1cm}

On the other hand, Theorem~\ref{20260606thm03} and its compact counterpart
Theorem~\ref{20260606prop01} do yield a complete answer to the
$\calQ_p$-Carleson measure problem for a large class of weights. Indeed, when the weight $W$ admits a Littlewood--Paley type derivative
norm equivalence, the embedding \eqref{20260607eqA11} can be reduced to the
corresponding embedding for $F'$ with a modified weight. This reduction allows
us to apply the dyadic trace theorem directly and obtain boundedness and
compactness criteria for such weights; see Section~\ref{Sec06}.
The fully general $\calQ_p$-Carleson measure problem appears to require a
refined dyadic gauge (as explained above), and we plan to investigate this direction in future work.

\end{rem}

The main task of this paper is to prove Theorem~\ref{20260606thm03}. Once the boundedness criterion is established, the corresponding compactness criterion follows by a standard vanishing modification of the same argument. For this purpose, we define the following ``truncated" version of \eqref{20260606eqF1}: for $0<\rho<1$ $a \in \D$, and $I \subseteq \T$, define
$$
        H_{a}^{\varphi,\rho}(I)
        :=\frac{1}{|I|^{p+2}}
        \int_{Q^{\textnormal{up}}_I\cap\{\rho<|\ze|<1\}}
        \calN_{p,\varphi,a}(\ze)dA(\ze).
$$

\begin{thm}[Compactness]\label{20260606thm02}
Let $0<p\le1$ and $\varphi$ be an analytic self-map of $\D$.  Then $C_\varphi:\calQ_p\to \calQ_p$
is compact if and only if $C_\varphi: \calQ_p \to \calQ_p$ is bounded and 
$$
\lim_{\rho\to1^-}\sup_{a\in\D} \sup_{\calD} \norm{H_{a}^{\varphi,\rho}}_{\calB_p(\calD)}=0, 
$$
where the second supremum is taken over all dyadic systems \(\calD\) on \(\T\).
\end{thm}

\vspace{0.1cm}

Finally, we summarize the main novelties of the present paper as follows.
\begin{enumerate}
    \item We replace pointwise estimates for counting functions by a dyadic
    trace-embedding formulation over Carleson tents.  This is essential in the
    range \(0<p<1\), where the \(p\)-Nevanlinna counting functions no longer
    enjoy the subharmonicity properties used in the classical BMOA theory.

    \item We introduce a dyadic $p$--capacity gauge \eqref{20260606eq04},
    which encodes the trace embedding through an equivalent dyadic capacity
    condition.  This converts the analytic embedding problem into a discrete
    extremal problem on the tree of Carleson tents.

    \item The proof of our main result Theorem \ref{20260606thm03} separates the upper and lower estimates in a way adapted
    to the analytic structure.  The upper estimate is obtained by deterministic
    non-analytic tent atoms together with the boundedness of the Bergman
    projection, while randomization is used only after projection to remove
    cross terms in the lower estimate.

    \item As an application of the dyadic trace theorem and its compact
    counterpart, we obtain boundedness and compactness criteria for the
    $\calQ_p$-Carleson measure problem for weights in the Littlewood--Paley
    class. This gives a large family of absolutely continuous measures for which the $\calQ_p$-Carleson measure problem can be reduced to the dyadic trace theorem treated in this paper.

\end{enumerate}

The rest of the paper is organized as follows.  In Section~\ref{Sec02}, we collect the
preliminary material needed for the proof of the dyadic trace embedding theorem.
This includes the dual formulation of the dyadic \(p\)--capacity gauge, the
boundedness of the weighted Bergman projection on \(\calT_p\), and several
auxiliary dyadic estimates over Carleson tents.  In Section~\ref{Sec03}, we prove
Theorem~\ref{20260606thm03}.
In Section~\ref{Sec04}, we derive the boundedness criterion for composition operators acting on $\calQ_p$, namely Theorem~\ref{20260606thm01}, from the trace embedding theorem and the
change-of-variables formula for the generalized \(p\)-Nevanlinna counting
functions.  In Section~\ref{Sec05}, we prove the compact trace embedding theorem
and then use its proof strategy to establish the compactness criterion for
\(C_\varphi\) on \(\calQ_p\), which gives Theorem~\ref{20260606thm02}.
Finally, in Section~\ref{Sec06}, we apply our main theorems to the $\calQ_p$-Carleson measure problem and obtain boundedness and compactness criteria for weights in the Littlewood--Paley class.

Throughout this paper, for $a ,b \in  \mathbb{R}$, $a\lesssim b$ means there exists a positive number $C$, which is independent of $a$ and $b$, such that $a\leq C\,b$. Moreover, if both $a \lesssim  b$ and $b\lesssim a$ hold, we say $a \simeq b$.
\\
\noindent{\bf Acknowledgement.}
The first author would like to thank his undergraduate advisor, Le Hai Khoi, for
bringing this problem to his attention around 2013, when the first author was
still a junior undergraduate student at Nanyang Technological University in Singapore.  Although the problem was not resolved at that time, this early exposure was one of the motivations for the present work. The authors would also like to thank Ruhan Zhao for his continued interest in
this problem and for many helpful comments and suggestions over the past decade. They also thank Jie Xiao for bringing to their attention several related works on $Q$-spaces in the Euclidean setting and for helpful comments and discussions. The first author was supported by the Simons Travel grant MPS-TSM-00007213.

\section{Preliminaries} \label{Sec02}

In this section, we collect several preliminary results needed for the proof of Theorem~\ref{20260606thm03}.

\subsection{Dual form of the dyadic $p$-capacity gauge}
We first prove a dual formulation for the dyadic \(p\)-capacity gauge defined in \eqref{20260606eq04}. The main result is stated as follows.

\begin{lem}\label{20260606lem01}
Let $h=\{h(I)\}_{I\in\calD}$ be nonnegative.  Then
\begin{equation}\label{20260606eq18}
\norm{h}_{\calB_p(\calD)}
=
\sup\left\{
        \sum_{I\in\calD}d_Ih(I):
        d_I\ge0,\quad
        \sup_{J\in\calD}\frac1{|J|^p}\sum_{I\subseteq J}d_I\le1
\right\}.
\end{equation}
\end{lem}

\begin{proof}
We first prove the result on a finite dyadic tree.  For \(N\in\mathbb N\), set $\calD_N:=\{I\in\calD: 2^{-N} \le |I| \le 1\}$. Define the truncated dyadic $p$-capacity gauge for $h$ by 
\begin{equation} \label{20260606eqA1}
\norm{h}_{\calB_p(\calD_N)}
:=
\inf\left\{
        \sum_{J\in\calD_N}\lambda_J |J|^p:
        \lambda_J\ge0,\ 
        \sum_{I \subseteq J, \; J \in \calD_N}\lambda_J\ge h(I)
        \ \textnormal{for all } I\in\calD_N
\right\}.
\end{equation} 
Since \(\calD_N\) is finite, this allows us to treat \eqref{20260606eqA1} a finite-dimensional primal problem:
\begin{center}
minimize $\mathsf J\left(\left\{\lambda_J\right\}_{J \in \calD_N}\right):=\sum\limits_{J\in\calD_N}\lambda_J |J|^p$ \\
subject to the constraints $\eta_I \left(\left\{\lambda_J\right\}_{J \in \calD_N}\right):=h(I)- \sum\limits_{I \subseteq J, \; J \in \calD_N}\lambda_J \le 0, \quad I \in \calD_N$. 
\end{center}
Its primal variables are \(\{\lambda_J\}_{J\in\calD_N}\). Therefore, the generalized Lagrange multiplier associated with the constraint indexed by $I \in \calD_N$ is given by $d_I \ge 0$, and the associated Lagrangian is
\begin{align} \label{20260606A2}
L\left(\{\lambda_J\}_{J \in \calD_N}, \left\{d_J \right\}_{J \in \calD_N}  \right)
&=\textsf J \left( \left\{\lambda_J \right\}_{J \in \calD_N} \right)+\sum_{I \in \calD_N} d_I \eta_I \left( \left\{\lambda_J \right\}_{J \in \calD_N} \right)  \nonumber \\ 
&=
\sum_{J\in\calD_N}\lambda_J|J|^p
+\sum_{I\in\calD_N}
d_I\left(h(I)-\sum_{I \subseteq J, \; J \in \calD_N}\lambda_J\right)   \nonumber \\
&=
\sum_{I\in\calD_N}d_Ih(I)
+
\sum_{J\in\calD_N}\lambda_J
\left(
|J|^p-\sum_{I \subseteq J, \; I \in \calD_N}d_I
\right).
\end{align}
For each fixed choice of $\left\{d_J\right\}_{J \in \calD_N}$, taking the infimum in \eqref{20260606A2} over \(\lambda_J\ge0\), we see that this infimum is finite precisely when
\[
        \sum_{I \subseteq J, \; I \in \calD_N}d_I\le |J|^p,
        \qquad J\in\calD_N.
\]
Under this condition, the infimum in $\left\{\lambda_J\right\}_{J \in \calD_N}$ equals
\[
        \sum_{I\in\calD_N}d_Ih(I).
\]
Therefore, by standard finite-dimensional linear-programming duality (see, e.g., \cite{Vanderbei2020}),
\begin{equation}\label{20260606eq-finite-duality}
\norm{h}_{\calB_p(\calD_N)}
=
\sup\left\{
        \sum_{I\in\calD_N}d_Ih(I):
        d_I\ge0,\quad
        \sum_{I \subseteq J, \; I \in \calD_N}d_I\le |J|^p
        \ \textnormal{for all } J\in\calD_N
\right\}.
\end{equation}
Finally, the desired identity \eqref{20260606eq18} follows by letting \(N\to\infty\) in \eqref{20260606eq-finite-duality}, since both sides are monotone increasing in \(N\).
\end{proof}

\subsection{Bergman projection on $\calT_p$}

The second ingredient in the proof of Theorem~\ref{20260606thm03} is the boundedness of the Bergman projection on the tent spaces \(\calT_p\). To begin with, for $0<p \le 1$, recall that the weighted Bergman projection $B_p$ associated with the measure $dA_p(z):=(1-|z|^2)^pdA(z)$ is defined as follows: for any $G$ being a measurable function on $\D$, 
$$
B_pG(z):=\int_{\D}\frac{G(w)}{(1-z\overline w)^{2+p}}dA_p(w),  \qquad z \in \D. 
$$

We have the following result. 

\begin{lem}\label{20260606lem02}
For $0<p\le1$,
$$
\norm{B_pG}_{\calT_p}\lesssim\norm{G}_{\calT_p}.
$$
\end{lem}

\begin{proof}
The proof of Lemma~\ref{20260606lem02} is routine, and the result is likely implicit in the literature. For the reader's convenience, we include the details. Without loss of generality, we may assume \(\left\|G\right\|_{\calT_p}=1\). For any $J \subseteq \T$, our goal is to show that
\begin{equation}\label{20260606eq21}
        \int_{Q_J}|B_pG(z)|^2\,dA_p(z)\lesssim |J|^p.
\end{equation}
To prove \eqref{20260606eq21}, we decompose \(G\) dyadically and carry out a ``local--far" analysis. This idea goes back to the seminal work of Fefferman and Stein~\cite{FS1972}, where they proved the conformal invariance of BMO on \(\D\); see also Garsia's presentation~\cite{Garsia1971} and Girela's lecture notes~\cite[Theorem~3.1]{Girela1999}. We now turn to the details.

\vspace{0.1cm}

Let \(J^{(n)}\subseteq\T\) be the arc with the same center as \(J\) and with length
\[
        |J^{(n)}|\simeq \min \{2^n|J|,1\}.
\]
Once \(J^{(n)}=\T\), we keep \(J^{(n)}\) equal to \(\T\) for all larger \(n\). Write
\[
        G_0=G\one_{Q_{J^{(2)}}},
        \qquad
        G_n=G\one_{Q_{J^{(n+1)}}\setminus Q_{J^{(n)}}},\quad n\ge2,
\]
and therefore, 
\begin{equation} \label{20260606A4}
B_pG(z)=B_pG_0(z)+\sum_{n \ge 2} B_pG_n(z), \qquad z \in \D. 
\end{equation} 

The local part is immediate from the \(L^2(dA_p)\)-boundedness of \(B_p\):
\begin{align*} 
\int_{Q_J}|B_pG_0(z)|^2\,dA_p(z)
&\le \int_{\D}|B_pG_0(z)|^2\,dA_p(z)  \nonumber \\
& \lesssim \int_{Q_{J^{(2)}}}|G(z)|^2\,dA_p(z) \nonumber \\
& =\left|J^{(2)} \right|^p \cdot \frac{1}{\left|J^{(2)}\right|^p } \int_{Q_{J^{(2)}}} |G(z)|^2\,dA_p(z) \nonumber \\
&\lesssim |J|^p,
\end{align*}
where in the last estimate above, we used the assumption $\left\|G \right\|_{\calT_p}=1$. 

We next treat the far part. Since
$$
        |1-z\overline w|
        \simeq
        (1-|z|)+(1-|w|)+\left|\frac{z}{|z|}-\frac{w}{|w|}\right|, \qquad z, w \in \D \backslash \{0\}
$$
then if $z \in Q_J$ and $w \in Q_{J^{(n+1)}} \backslash Q_{J^{(n)}}$, one has
$$
        |1-z\overline w|\gtrsim \min \{2^n|J|,1\}.
$$
By Cauchy--Schwarz, 
\begin{align} \label{20260606eqA3}
        |B_pG_n(z)|
        & \le \int_{Q_{J^{(n+1)}}\setminus Q_{J^{(n)}}} \frac{|G(w)|}{\left| 1-z \overline{w} \right|^{2+p}}dA_p(w) \nonumber  \\
        &\lesssim  \frac{1}{\left(\min \{2^n|J|,1\}\right)^{2+p}}\int_{Q_{J^{(n+1)}}}|G(w)|\,dA_p(w)\nonumber \\
        &\le \frac{1}{\left(\min \{2^n|J|,1\}\right)^{2+p}}
        \left(\int_{Q_{J^{(n+1)}}}|G(z)|^2\,dA_p(z)\right)^{1/2}
        A_p(Q_{J^{(n+1)}})^{1/2}.
\end{align}
Since \(\left\|G\right\|_{\calT_p}=1\), we have
$$
        \int_{Q_{J^{(n+1)}}}|G(z)|^2\,dA_p(z)\lesssim \left|J^{(n+1)} \right|^p \simeq \left(\min \{2^n|J|,1\}\right)^p
$$
        and moreover, 
$$
A_p(Q_{J^{(n+1)}})=\int_{Q_{J^{(n+1)}}} (1-|z|^2)^pdA(z) \lesssim \left(\min \{2^n|J|,1\}\right)^{p+2}.
$$
Plugging the above estimates back to \eqref{20260606eqA3}, we deduce that  
$$
        |B_pG_n(z)|\lesssim \frac{1}{\min \{2^n|J|,1 \}}, \qquad z\in Q_J.
$$
Therefore, if $2^n|J|<1$, then
\begin{align*}
\left\|B_pG_n\right\|_{L^2(Q_J,dA_p)}
&\lesssim (2^n|J|)^{-1}A_p(Q_J)^{1/2} \\
&\simeq  (2^n|J|)^{-1} \cdot |J|^{\frac{p+2}{2}} \\
&= 2^{-n}|J|^{p/2}.
\end{align*}
If $2^n|J| \ge 1$, then there are $O(1)$ many terms contained in the sum \eqref{20260606A4}, which satisfies the estimate
$$
\left\|B_pG_n\right\|_{L^2(Q_J,dA_p)} \lesssim A_p(Q_J)^{1/2} \simeq |J|^{\frac{p+2}{2}} \le |J|^{\frac{p}{2}}. 
$$
Hence
\begin{align*}
\int_{Q_J}\left| \sum_{n \ge 2} B_pG_n(z) \right |^2\,dA_p(z)  
& \le \left[\sum_{n\ge2}\left\|B_pG_n\right\|_{L^2(Q_J,dA_p)} \right]^2 \\
& \lesssim |J|^p,
\end{align*}
which finishes the proof of the far part.

The desired conclusion follows by combining the local and far estimates.
\end{proof}

\subsection{More dyadic analysis on Carleson tents}

We still need two auxiliary results before proving Theorem~\ref{20260606thm03}. The first result is a further decomposition of upper Carleson tents. Let \(\eta>0\) be sufficiently small. For each \(I\in\calD\), choose a finite partition
\begin{equation} \label{20260606C1}
Q_I^{\textnormal{up}}=\bigcup_{\nu=1}^{N_0}Q_{I,\nu}
\end{equation} 
where \(N_0=N_0(\eta)\) is independent of \(I\), each \(Q_{I,\nu}\) is a ``truncated" Whitney-type tent with Euclidean diameter at most \(\eta |I|\), and
$$
        A_p(Q_{I,\nu})\simeq |I|^{p+2}.
$$
Here the implicit constants may depend on \(\eta\), which is fixed once and for all. We choose \(\eta\) sufficiently small (depending on $p$) so that, on each product \(Q_{I,\nu}\times Q_{I,\nu}\), the Bergman kernel has essentially constant argument.

Define the $L^2(dA_p)$--normalized atom
\begin{equation} \label{2026060eqC2}
        b_{I,\nu}:=A_p(Q_{I,\nu})^{-1/2}\one_{Q_{I,\nu}}.
\end{equation} 
Then \(\norm{b_{I,\nu}}_{L^2(dA_p)}=1\).

\begin{lem}\label{20260606lem03}
For every \(I\in\calD\) and every \(1\le\nu\le N_0\),
\begin{equation}\label{20260606eq28}
        |B_pb_{I,\nu}(z)|\gtrsim |I|^{-(p+2)/2},
        \qquad z\in Q_{I,\nu}.
\end{equation}
\end{lem}

\begin{proof}
Fix \(z\in Q_{I,\nu}\). Since \( Q_{I,\nu} \subset Q_I^{\textnormal{up}}\), we have for every \(w\in Q_{I,\nu}\),
\[
        |1-z\overline w|\simeq |I|.
\]
Moreover, for any $w, w_0 \in Q_{I,\nu}$, one has 
\[
        \left|
        \frac{1-z\overline w}{1-z\overline {w_0}}-1
        \right|
        =
        \frac{|z|\,|w-w_0|}{|1-z\overline {w_0}|}
        \lesssim
        \frac{\operatorname{diam}(Q_{I,\nu})}{|I|}
        \lesssim \eta .
\]
Choosing \(\eta>0\) sufficiently small, it follows that the arguments of
\[
        (1-z\overline w)^{-2-p},\qquad w\in Q_{I,\nu},
\]
are contained in an interval of length at most \(\pi/10\). Hence, 
\[
        \left|
        \int_{Q_{I,\nu}} \frac{1}{(1-z\overline w)^{2+p}}\,dA_p(w)
        \right|
        \gtrsim
        \int_{Q_{I,\nu}} \frac{1}{|1-z\overline w|^{2+p}}\,dA_p(w).
\]
This further gives
\begin{align*}
        |B_pb_{I,\nu}(z)|
        &=
        A_p(Q_{I,\nu})^{-1/2}
        \left|
        \int_{Q_{I,\nu}} \frac{1}{(1-z\overline w)^{2+p}}\,dA_p(w)
        \right|  \\
        &\gtrsim
        A_p(Q_{I,\nu})^{-1/2}
        \int_{Q_{I,\nu}} \frac{1}{|1-z\overline w|^{2+p}}\,dA_p(w) \\
        &\simeq
        A_p(Q_{I,\nu})^{-1/2}|I|^{-2-p}A_p( Q_{I,\nu}) \\
        &=
        A_p(Q_{I,\nu})^{1/2}|I|^{-2-p}
        \simeq
        |I|^{(p+2)/2}|I|^{-2-p} \\
        &=
        |I|^{-(p+2)/2}.
\end{align*}
This proves \eqref{20260606eq28}.
\end{proof}

We also need the following packing lemma, which allows us to pass from a dyadic packing condition to its continuous analogue.

\begin{lem}\label{20260606lem04}
Let \(\{d_I\}_{I\in\calD}\) be a nonnegative sequence satisfying
\begin{equation}\label{20260606eq29}
        \sup_{J\in\calD}\frac1{|J|^p}\sum_{I\subseteq J}d_I\le1.
\end{equation}
Let \(K\subseteq\T\) be an arbitrary arc, and let
\[
        \mathcal I(K):=\{I\in\calD:Q_I^{\textnormal{up}}\cap Q_K\ne\varnothing\}.
\]
Then
\begin{equation}\label{20260606eq30}
        \sum_{I\in\mathcal I(K)}d_I\lesssim |K|^p.
\end{equation}
\end{lem}

\begin{proof}
Observe that if \(Q_I^{\textnormal{up}}\cap Q_K\ne\varnothing\), then \(|I|\lesssim |K|\). We consider two different cases. 

We first consider those \(I\in\mathcal I(K)\) with \(I\subseteq K\). Let \(\mathcal M(K)=\{J\} \) denote the collection of maximal dyadic subintervals of \(K\). Then the intervals in \(\mathcal M(K)\) are pairwise disjoint, and every dyadic interval \(I\subseteq K\) is contained in a unique \( J \in\mathcal M(K)\). Moreover, by\eqref{20260606eq29},
\[
        \sum_{I\subseteq J}d_I\le |J|^p.
\]
The maximal dyadic subintervals of an arbitrary arc form at most two Whitney chains (around two endpoints) plus one central interval (see, Figure \ref{Fig1}). Their lengths decrease geometrically; hence
\[
       \sum_{I \subseteq K, \; I \in \calI(K)} d_I=\sum_{J \in \calM(K)} \sum_{I \subseteq J}d_I \le  \sum_{J \in \calM(K)} |J|^p\lesssim |K|^p.
\]

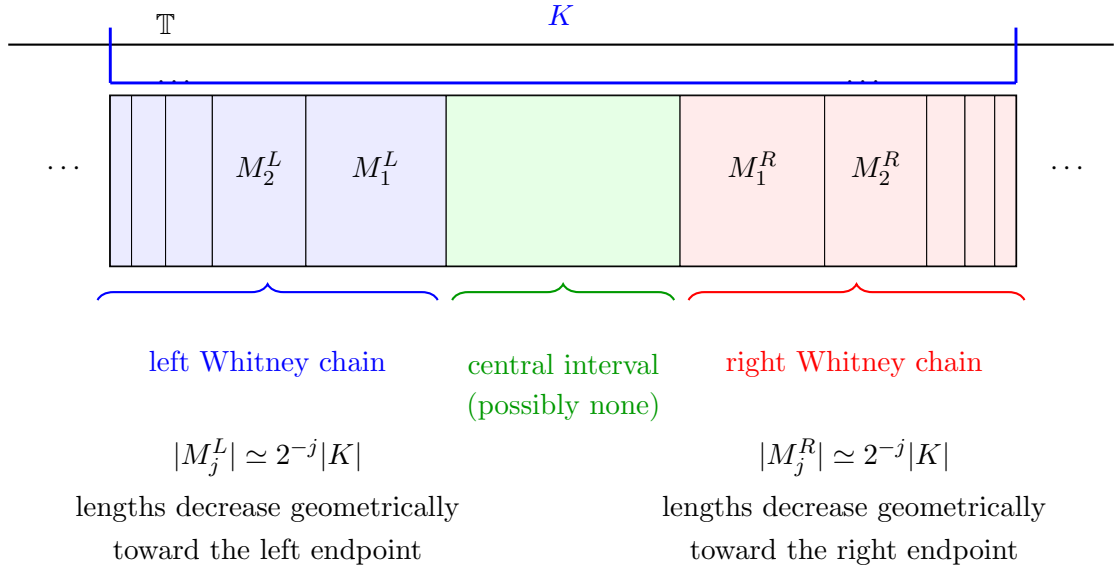
\begin{figure}[ht]
\centering
\begin{tikzpicture}[x=0.75cm,y=0.75cm, scale=1.5pt, font=\small]

\draw[line width=0.8pt] (1.5,5) -- (14.5,5);
\node[above] at (3.35,5) {$\mathbb T$};

\node at (3.45,4.55) {\(\cdots\)};
\node at (11.55,4.55) {\(\cdots\)};

\draw[blue,line width=1.1pt] (2.7,4.55) -- (13.35,4.55);
\draw[blue,line width=1.1pt] (2.7,4.55) -- (2.7,5.2);
\draw[blue,line width=1.1pt] (13.35,4.55) -- (13.35,5.2);
\node[blue] at (8.0,5.35) {\(K\)};

\draw[line width=0.9pt] (2.7,2.4) rectangle (13.35,4.4);

\filldraw[fill=blue!8,draw=black] (2.7,2.4) rectangle (2.95,4.4);
\filldraw[fill=blue!8,draw=black] (2.95,2.4) rectangle (3.35,4.4);
\filldraw[fill=blue!8,draw=black] (3.35,2.4) rectangle (3.90,4.4);
\filldraw[fill=blue!8,draw=black] (3.90,2.4) rectangle (5.00,4.4);
\filldraw[fill=blue!8,draw=black] (5.00,2.4) rectangle (6.65,4.4);

\filldraw[fill=green!10,draw=black] (6.65,2.4) rectangle (9.40,4.4);

\filldraw[fill=red!8,draw=black] (9.40,2.4) rectangle (11.10,4.4);
\filldraw[fill=red!8,draw=black] (11.10,2.4) rectangle (12.30,4.4);
\filldraw[fill=red!8,draw=black] (12.30,2.4) rectangle (12.75,4.4);
\filldraw[fill=red!8,draw=black] (12.75,2.4) rectangle (13.10,4.4);
\filldraw[fill=red!8,draw=black] (13.10,2.4) rectangle (13.35,4.4);


\node at (4.45,3.55) {\(M_2^L\)};
\node at (5.82,3.55) {\(M_1^L\)};
\node at (10.25,3.55) {\(M_1^R\)};
\node at (11.70,3.55) {\(M_2^R\)};

\node at (2.15,3.55) {\(\cdots\)};
\node at (13.95,3.55) {\(\cdots\)};

\draw[decorate,decoration={brace,amplitude=6pt},blue,line width=0.8pt]
(2.55,2.0) -- (6.55,2.0);
\node[blue] at (4.55,1.25) {left Whitney chain};

\draw[decorate,decoration={brace,amplitude=6pt},green!60!black,line width=0.8pt]
(6.72,2.0) -- (9.33,2.0);
\node[green!60!black] at (8.03,1.25) {central interval};
\node[green!60!black] at (8.03,0.75) {(possibly none)};

\draw[decorate,decoration={brace,amplitude=6pt},red,line width=0.8pt]
(9.50,2.0) -- (13.45,2.0);
\node[red] at (11.45,1.25) {right Whitney chain};

\node at (4.55,0.15) {\(\lvert M_j^L\rvert \simeq 2^{-j}\lvert K\rvert\)};
\node at (4.55,-0.45) {lengths decrease geometrically};
\node at (4.55,-0.95) {toward the left endpoint};

\node at (11.45,0.15) {\(\lvert M_j^R\rvert \simeq 2^{-j}\lvert K\rvert\)};
\node at (11.45,-0.45) {lengths decrease geometrically};
\node at (11.45,-0.95) {toward the right endpoint};

\end{tikzpicture}
\caption{Maximal dyadic subintervals of an arbitrary arc \(K\).}
\label{Fig1}
\end{figure}

Next, we consider those \(I\in\mathcal I(K)\) with \(I\not\subseteq K\). Such an interval must meet one of the two endpoints of \(K\), and hence, at each dyadic scale, there are at most two such intervals. Moreover, all such intervals satisfy \(|I|\lesssim |K|\). By \eqref{20260606eq29}, we have \(d_I\le |I|^p\), and therefore their total contribution is bounded by
\[
        \sum_{m\ge0}(2^{-m}|K|)^p\lesssim |K|^p.
\]
This proves \eqref{20260606eq30}.
\end{proof}

\section{Proof of  Theorem~\ref{20260606thm03}}\label{Sec03}

\subsection{Sufficiency}

Assume that there exists a dyadic system \(\calD\) on \(\T\) such that $\norm{h_W}_{\calB_p(\calD)}<\infty$. Let \(F\in \calQ_p\) and fix any dyadic system $\calD$ on $\T$. Since $\{Q_I^{\textnormal{up}}\}_{I \in \calD}$ gives a disjoint partition of $\D$, we have
\begin{equation} \label{20260606eqB1}
\int_{\D}|F'(z)|^2W(z)\,dA(z)=\sum_{I\in\calD}\int_{Q_I^{\textnormal{up}}}|F'(z)|^2W(z)\,dA(z).
\end{equation} 
By the sub-mean value property for holomorphic functions (see, e.g., \cite[Lemma 2.24]{Zhu2005}), 
\begin{align} \label{20260606eqB2}
\int_{Q_I^{\textnormal{up}}}|F'(z)|^2W(z)\,dA(z)
&\le \left(\sup_{z \in Q_I^{\textnormal{up}}}|F'(z)|^2 \right) \cdot 
\int_{Q_I^{\textnormal{up}}}W(z)\,dA(z) \nonumber  \\
&\lesssim \left(\frac1{|I|^{p+2}} \int_{\widetilde Q_I^{\textnormal{up}}}|F'(z)|^2\,dA_p(z)\right)
\cdot \left(\int_{Q_I^{\textnormal{up}}}W(z)\,dA(z) \right)  \nonumber  \\
&=\left( \int_{\widetilde Q_I^{\textnormal{up}}}|F'(z)|^2\,dA_p(z)\right)
\cdot \left(\frac1{|I|^{p+2}} \int_{Q_I^{\textnormal{up}}}W(z)\,dA(z) \right) \nonumber  \\
&=h_W(I)d_I(F),
\end{align}
where \(\widetilde Q_I^{\textnormal{up}}\) denotes a fixed slight enlargement\footnote{For example, \(\widetilde Q_I^{\textnormal{up}}\) may be chosen with the same angular and radial centers as \(Q_I^{\textnormal{up}}\), and with both angular length and radial thickness enlarged by a fixed factor \(c_0>1\), say \(c_0=1.1\). The precise choice is immaterial as long as it is fixed uniformly in \(I\).} of \(Q_I^{\textnormal{up}}\), and
\[
        d_I(F):=\int_{\widetilde Q_I^{\textnormal{up}}}|F'(z)|^2\,dA_p(z).
\]
Observe that the enlarged tents in $\{\widetilde Q_I^{\textnormal{up}}\}_{I \in \calD}$ have finite overlap. Therefore, for every \(J\in\calD\), by \eqref{20260606eq01}, one has
\begin{align} \label{20260606eqB3}
\sum_{I\subseteq J}d_I(F)
& \lesssim \int_{\widetilde Q_J} |F'(z)|^2\,dA_p(z) \nonumber \\
& =|J|^p \cdot \frac{1}{|J|^p} \int_{\widetilde Q_J} |F'(z)|^2\,dA_p(z) \nonumber \\
& \lesssim |J|^p\norm{F}_{\calQ_p,*}^2,
\end{align}
where $\widetilde Q_J:=\bigcup_{I \subseteq J, \; I \in \calD} \widetilde Q_I^{\textnormal{up}}$. By \eqref{20260606eqB1}, \eqref{20260606eqB2}, and Lemma~\ref{20260606lem01} with the constraint
$$
\sup_{J \in \calD} \frac{1}{|J|^p} \left( \sum_{I \subseteq J} \frac{d_I(F)}{\left\|F \right\|^2_{\calQ_p, *}} \right) \lesssim 1
$$
(which is guaranteed by \eqref{20260606eqB3}), 
we have
\begin{align*}
\int_{\D} |F'(z)|^2W(z)dA(z)
& \lesssim \sum_{I \in \calD} h_W(I)d_I(F) \\
&   \lesssim      \norm{h_W}_{\calB_p(\calD)}\norm{F}_{\calQ_p,*}^2.
\end{align*}
This proves \eqref{20260606eq15}.

\subsection{Necessity}

Assume that \eqref{20260606eq15} holds with best constant \(M\). By Lemma~\ref{20260606lem01}, it is enough to show that for any dyadic system $\calD$ on $\T$ and any finitely supported non-negative sequence \(\{d_I\}_{I\in\calD}\) satisfying 
$$
\sup_{J\in\calD}\frac1{|J|^p}\sum_{I\subseteq J}d_I\le1,
$$
one has
\begin{equation}\label{20260606eq31}
        \sum_{I\in\calD}d_Ih_W(I)\lesssim M.
\end{equation}
Let \(\varepsilon=\{\eps_{I,\nu}\}_{\substack{I \in \calD \\ 1 \le \nu \le N_0}}\) be independent Rademacher random variables and set
\begin{equation}\label{20260606eq32}
        G_\eps(w):=\sum_{I \in \calD} \sum_{\nu=1}^{N_0} \eps_{I,\nu}d_I^{1/2}b_{I,\nu}(w), \qquad w \in \D. 
\end{equation}
Here, recall that $N_0 \in \N$ is defined as in \eqref{20260606C1} and $b_{I,\nu}$ is the $L^2(dA_p)$--normalized atom associated to the ``truncated" Whitney-type tent $Q_{I, \nu}$, defined as in \eqref{20260606C1}. Since $\{d_I\}_{I \in \calD}$ is finitely supported, the sum \eqref{20260606eq32} is finite. Moreover, since the boxes \(Q_{I,\nu}\) are pairwise disjoint, 
$$
|G_\eps(w)|^2=
\sum_{I \in \calD} \sum_{\nu=1}^{N_0} d_I|b_{I,\nu}(w)|^2, \qquad w \in \D.
$$
Therefore, for any arc \(K \subseteq \T\),
\begin{align*}
        \int_{Q_K}|G_\eps(w)|^2\,dA_p(w)
        &\le
        \sum_{I \in \calD, \; Q_I^{\textnormal{up}}\cap Q_K\ne\varnothing}
        \sum_{\nu=1}^{N_0}
        d_I\int_{Q_{I,\nu}}|b_{I,\nu}(w)|^2\,dA_p(w) \\
        &\lesssim
        \sum_{I\in\mathcal I(K)}d_I
        \lesssim
        |K|^p,
\end{align*}
where the last estimate above follows from Lemma~\ref{20260606lem04}. Hence
$$
\norm{G_\eps}_{\calT_p}^2\lesssim 1 \qquad\text{uniformly in }\eps.
        $$
By Lemma~\ref{20260606lem02},
$$
\norm{B_pG_\eps}_{\calT_p}^2\lesssim1 \qquad\text{uniformly in }\eps.
$$
Define
$$
F_\eps(z):=\int_0^z B_pG_\eps(\xi)\,d\xi, \qquad z \in \D. 
$$
Then \(F_\eps\in \calQ_p\) and
\begin{equation}\label{20260606eq35}
        \norm{F_\eps}_{\calQ_p,*}^2
        \simeq
        \norm{B_pG_\eps}_{\calT_p}^2
        \lesssim1,
\end{equation}
where we have used \eqref{20260606D1}. Applying \eqref{20260606eq15} to \(F_\eps\) gives
\begin{equation}\label{20260606eq36}
\int_{\D} |F'_\varepsilon(z)|^2W(z)dA(z)=\int_{\D}|B_pG_\eps(z)|^2W(z)\,dA(z)\lesssim M\norm{F_\eps}_{\calQ_p,*}^2 \lesssim M
\end{equation}
for every choice of signs. Taking expectation and using the assumption that \(\varepsilon=\{\eps_{I,\nu}\}_{\substack{I \in \calD \\ 1 \le \nu \le N_0}}\) are independent Rademacher random variables, 
\begin{equation}\label{20260606eq37}
        \mathbb E_\eps |B_pG_\eps(z)|^2
        =
        \sum_{I \in \calD} \sum_{\nu=1}^{N_0} d_I|B_pb_{I,\nu}(z)|^2.
\end{equation}
Thus, by \eqref{20260606eq36} and \eqref{20260606eq37}
\begin{align*}
        M
        & =\mathbb E_\eps M \gtrsim \mathbb E_\eps \left(\int_{\D}|B_pG_\eps(z)|^2W(z)\,dA(z) \right) \\
        &=
       \sum_{I \in \calD} \sum_{\nu=1}^{N_0}d_I
        \int_{\D}|B_pb_{I,\nu}(z)|^2W(z)\,dA(z) \\
        &\ge
        \sum_{I \in \calD} \sum_{\nu=1}^{N_0}d_I
        \int_{Q_{I,\nu}}|B_pb_{I,\nu}(z)|^2W(z)\,dA(z).
\end{align*}
By Lemma~\ref{20260606lem03},
\[
        \int_{Q_{I,\nu}}|B_pb_{I,\nu}(z)|^2W(z)\,dA(z)
        \gtrsim
    \frac{1}{|I|^{p+2}}\int_{Q_{I,\nu}}W(z)\,dA(z).
\]
Summing in \(\nu\) gives
\begin{equation} \label{20260606eqG1}
        M
        \gtrsim
        \sum_{I \in \calD} \frac{d_I}{|I|^{p+2}}
        \int_{Q_I^{\textnormal{up}}}W(z)\,dA(z)
        =
        \sum_I d_Ih_W(I).
\end{equation} 
This proves \eqref{20260606eq31}, and the proof of Theorem~\ref{20260606thm03} is complete.

\medskip 

\section{Proof of Theorem \ref{20260606thm01}: boundedness} \label{Sec04}

This section is devoted to proving Theorem \ref{20260606thm01}. First, we recall the following change variable formula.  

\begin{lem}\label{20260606lem05}
For any $f \in H(\D)$, any analytic self-map $\varphi$ of $\D$, and any $a\in\D$,
\begin{equation} \label{20260606eqE2}
        \int_{\D}|(f\circ\varphi)'(z)|^2g(z,a)^pdA(z)
        =
        \int_{\D}|f'(\ze)|^2\calN_{p,\varphi,a}(\ze) dA(\ze).
\end{equation} 
\end{lem}

\begin{proof}
By the chain rule,
\begin{align*}
\textnormal{LHS of \eqref{20260606eqE2}}
&=\int_{\D}|f'(\varphi(z))|^2|\varphi'(z)|^2g^p(z,a) dA(z)  \\
&=\int_{\D}|f'(\ze)|^2 \sum_{\varphi(z)=\ze}g(z,a)^p dA(\ze) \\
&=\int_{\D}|f'(\ze)|^2\calN_{p,\varphi,a}(\ze) dA(\ze).
\end{align*}
The proof is complete.
\end{proof}

\begin{proof}[Proof of Theorem \ref{20260606thm01}]
Assume first that \eqref{20260606eq06} holds for some dyadic system $\calD$ on $\T$.  Let $f\in \calQ_p$.  For each $a\in\D$, apply Theorem \ref{20260606thm03} with $W=\calN_{p,\varphi,a}$ to deduce that 
\[
        \int_{\D}|f'(\ze)|^2\calN_{p,\varphi,a}(\ze) dA(\ze)
        \lesssim \norm{h_{a}^{\varphi}}_{\calB_p(\calD)}
        \norm{f}_{\calQ_p,*}^2.
\]
Taking the supremum in $a$ and using Lemma \ref{20260606lem05}, we obtain
\[
        \norm{C_\varphi f}_{\calQ_p,*}^2
        \lesssim
        \left(\sup_{a \in \D} \norm{h_{a}^{\varphi}}_{\calB_p(\calD)}\right) \cdot 
        \norm{f}_{\calQ_p,*}^2.
\]
This proves boundedness with respect to the \(\calQ_p\)-seminorm. The passage from the seminorm to the full \(\calQ_p\)-norm follows from the standard point-evaluation estimate in \(\calQ_p\).

\vspace{0.1cm}

Conversely, assume $C_\varphi:\calQ_p\to \calQ_p$ is bounded, and let its seminorm operator norm be $L$.  Fix $a\in\D$.  Lemma \ref{20260606lem05} gives, for every $f\in \calQ_p$,
\[
        \int_{\D}|f'(\ze)|^2\calN_{p,\varphi,a}(\ze) dA(\ze)
        =\int_{\D}|(f\circ\varphi)'(z)|^2g(z,a)^pdA(z)
        \le L^2\norm{f}_{\calQ_p,*}^2.
\]
Theorem \ref{20260606thm03} therefore gives for any dyadic system $\calD$ on $\T$, 
\[
        \norm{h_{a}^{\varphi}}_{\calB_p(\calD)}\lesssim L^2.
\]
Taking the supremum in $a$ proves \eqref{20260606eq06}.  Finally, \eqref{20260606eq07} follows clearly from the above argument. The proof is complete.
\end{proof}

\medskip 

\section{Compactness of composition operators} \label{Sec05}

In this section we prove the compactness part of our main result.  We first
record the compact version of Theorem~\ref{20260606thm03}, and then adapt its
argument to give a proof for Theorem~\ref{20260606thm02}.

\subsection{Compact trace embeddings}

We shall use the following standard compactness criterion, whose proof is standard and therefore omitted. 

\begin{lem}\label{20260606lem06} 
Let \(X\subset H(\D)\) be a Banach space of analytic functions such that the inclusion 
\[ X\hookrightarrow H(\D) \] 
is continuous, where \(H(\D)\) is equipped with the compact--open topology. 

Let \(T:\calQ_p\to X\) be a bounded linear operator. Then \(T\) is compact if and only if \[ \|Tf_n\|_X\to0 \] for every bounded sequence \(\{f_n\}\subset\calQ_p\) such that \(f_n\to0\) uniformly on compact subsets of \(\D\).
\end{lem}

For a nonnegative locally integrable weight \(W\) on \(\D\) and \(0<\rho<1\),
set
\[
        W^\rho(z):=\one_{\{|z|>\rho\}}W(z).
\]
For any $I \subseteq \T$, define
\[
h_{W^\rho}(I):=\frac{1}{|I|^{p+2}}
\int_{Q_I^{\textnormal{up}}}W^\rho(z)\,dA(z).
\]

We have the following compact version of Theorem \ref{20260606thm03}.

\begin{thm}\label{20260606prop01}
Let \(0<p\le1\), and let \(W\ge0\) be locally integrable on \(\D\).  Then the mapping
\[
       \frac{d}{dz}: \calQ_p\longrightarrow L^2(WdA), \qquad f \mapsto f'
\]
is compact if and only if 
\begin{equation}\label{20260606eq39}
\lim_{\rho\to1^-} \sup_{\calD} \big\|h_{W^\rho}\big\|_{\calB_p(\calD)}=0,
\end{equation}
where the supremum is taken over all dyadic systems \(\calD\) on \(\T\).
\end{thm}

\begin{proof}
Assume first that \eqref{20260606eq39} holds. Let \(\{f_n\}\subseteq\calQ_p\) be a bounded sequence such that \(f_n\to0\) uniformly on compact subsets of \(\D\). By Cauchy's integral formula, we also have \(f_n'\to0\) uniformly on compact subsets of \(\D\).

Fix \(\epsilon>0\). Choose \(\rho<1\) sufficiently close to \(1\) so that
\[
\sup_{\calD}\big\|h_{W^\rho}\big\|_{\calB_p(\calD)}
        <\epsilon .
\]
By Theorem \ref{20260606thm03} applied to \(W^\rho\), we obtain for any $n \ge 1$
\[
\int_{\{|z|>\rho\}} |f_n'(z)|^2 W(z)\,dA(z)
=\int_{\D}|f_n'(z)|^2W^\rho(z)\,dA(z)
        \lesssim
        \varepsilon\,\sup_{n \ge 1}\|f_n\|_{\calQ_p,*}^2 \lesssim \epsilon. 
\]
On the other hand, since \(W\) is locally integrable and \(f_n'\to0\)
uniformly on compact subsets,
\[
        \int_{\{|z|\le\rho\}} |f_n'(z)|^2 W(z)\,dA(z)\to0 \qquad \textrm{as} \ n \to \infty.
\]
It follows that
\[
        \int_{\D}|f_n'(z)|^2 W(z)\,dA(z)\to0,
\]
and hence the $\frac{d}{dz}: \calQ_p \to L^2(WdA)$ is compact.

\vspace{0.1cm}

Conversely, assume that $\frac{d}{dz}: \calQ_p \to L^2(WdA)$ is compact and assume that
\eqref{20260606eq39} fails.  Then there exist \(\rho_n\to1\), dyadic systems
\(\calD_n\), and a constant \(c_0>0\) such that
$$
\big\|h_{W^{\rho_n}}\big\|_{\calB_p(\calD_n)}
\ge c_0 .
$$
By Lemma~\ref{20260606lem01}, this means that for each \(n \ge 1\) there exists a finitely supported nonnegative sequence
\(\{d_I^{(n)}\}_{I\in\calD_n}\) such that
\begin{equation}\label{20260606eqE6}
\sup_{J\in\calD_n} \frac1{|J|^p} \sum_{I\subseteq J}d_I^{(n)}\le 1
\end{equation} 
and
\begin{equation}\label{20260606eq41}
       \sum_{I \in \calD_n} d^{(n)}_I h_{W^{\rho_n}}(I)= \sum_{I\in\calD_n} \frac{d_I^{(n)}}{|I|^{p+2}}
        \int_{Q_I^{\textnormal{up}}} W^{\rho_n}(z)\,dA(z)
        \ge \frac{c_0}{2}.
\end{equation}
Moreover, we may assume that  $d_I^{(n)}=0$
    whenever $Q_I^{\textnormal{up}}\cap\{|z|>\rho_n\}=\emptyset$, 
since removing such terms does not change the left-hand side of
\eqref{20260606eq41} and preserves the dyadic \(p\)-Carleson packing condition \eqref{20260606eqE6}. 

We now repeat the randomized Bergman-projection construction from the proof of the necessary part of Theorem~\ref{20260606thm03}, with \(\calD=\calD_n\),
\(d_I=d_I^{(n)}\), and \(W=W^{\rho_n}\).  More precisely, let
\(\varepsilon=\{\varepsilon_{I,\nu}\}_{\substack{I \in \calD_n \\ 1 \le \nu \le N_0}}\) be independent Rademacher random
variables and set
\[
        G_{n,\varepsilon}
        :=
        \sum_{I\in\calD_n}\sum_{\nu=1}^{N_0}
        \varepsilon_{I,\nu}\big(d_I^{(n)}\big)^{1/2}b_{I,\nu}.
\]
Define
\[
        F_{n,\varepsilon}(z)
        :=
        \int_0^z B_pG_{n,\varepsilon}(\xi)\,d\xi .
\]
As in the proof of Theorem~\ref{20260606thm03}, the dyadic packing condition \eqref{20260606eqE6} and Lemma \ref{20260606lem04} imply
\[
        \|F_{n,\varepsilon}\|_{\calQ_p,*}\lesssim1
\]
uniformly in \(n\) and in the choice of signs.  Moreover, after taking
expectation in the signs, the same argument gives
\[
        \mathbb E_\varepsilon
        \int_{\D}|F_{n,\varepsilon}'(z)|^2W^{\rho_n}(z)\,dA(z)
        \gtrsim
        \sum_{I\in\calD_n}
        d_I^{(n)}h_{W^{\rho_n}}(I).
\]
Combining this with \eqref{20260606eq41}, we obtain
\[
        \mathbb E_\varepsilon
        \int_{\D}|F_{n,\varepsilon}'(z)|^2W^{\rho_n}(z)\,dA(z)
        \gtrsim1.
\]
Hence, for each \(n\), there exists a choice of signs
\(\varepsilon^{(n)}\) such that
\[
        \int_{\D}|F_{n,\varepsilon^{(n)}}'(z)|^2W^{\rho_n}(z)\,dA(z)
        \gtrsim1.
\]
We now set
\[
        F_n:=F_{n,\varepsilon^{(n)}} .
\]
Then \(F_n(0)=0\), and
\begin{equation}\label{20260606eq42}
        \|F_n\|_{\calQ_p,*}\lesssim1,
        \qquad
        \int_{\D}|F_n'(z)|^2W^{\rho_n}(z)\,dA(z)\gtrsim1 .
\end{equation}

It remains to check that \(F_n\to0\) uniformly on compact subsets of \(\D\).  Let \(G_n:=G_{n, \varepsilon^{(n)}}\)
be the non-analytic function used before applying the Bergman projection, so
that
\[
        F_n'=B_pG_n .
\]
If \(Q_I^{\textnormal{up}}\cap\{|z|>\rho_n\}\ne\emptyset\), then
\begin{equation} \label{20260606G60}
        Q_I^{\textnormal{up}}
        \subset
        \{|z|>1-C_0(1-\rho_n)\}
\end{equation} 
for an absolute constant \(C_0 \ge 1\).  Therefore, since we assume that $d_I^{(n)}=0$ whenever $Q_I^{\textnormal{up}} \cap \{|z|>\rho_n \}=\emptyset$, 
\begin{equation} \label{20260606G61}
        \supp \; G_n
        \subset
        \{|z|>1-C_0(1-\rho_n)\},
\end{equation}
Moreover, the dyadic packing condition \eqref{20260606eqE6} and Lemma \ref{20260606lem04} gives
\[
        \|G_n\|_{\calT_p}\lesssim1,
\]
and in particular
\[
        \int_{\D}|G_n(w)|^2\,dA_p(w)\lesssim1.
\]

Let \(K\subseteq\D\) be any compact subset.  For \(z\in K\), the Bergman kernel
\((1-z\overline w)^{-2-p}\) is uniformly bounded.  Hence, by Cauchy--Schwarz,
\begin{align*}
        |F_n'(z)|
        =
        |B_pG_n(z)|
        &\lesssim_K
        \int_{\{|w|>1-C_0(1-\rho_n)\}} |G_n(w)|\,dA_p(w)  \\
        &\le
        \left(\int_{\D}|G_n(w)|^2\,dA_p(w)\right)^{1/2}
        A_p\big(\{|w|>1-C_0(1-\rho_n)\}\big)^{1/2}  \\
        &\lesssim
        (1-\rho_n)^{(p+1)/2}\to0 .
\end{align*}
Thus \(F_n'\to0\) uniformly on compact subsets.  Since \(F_n(0)=0\), it also
follows that \(F_n\to0\) uniformly on compact subsets of \(\D\).

However, by Lemma \ref{20260606lem06}, this contradicts the compactness of the mapping $\frac{d}{dz}: \calQ_p \to L^2(WdA)$, because
\eqref{20260606eq42} gives
\[
        \int_{\D}|F_n'(z)|^2W(z)\,dA(z)
        \ge
        \int_{\D}|F_n'(z)|^2W^{\rho_n}(z)\,dA(z)
        \gtrsim1.
\]
The proof is complete.
\end{proof}

\subsection{Compactness of composition operators acting on $\calQ_p$}

We now prove the compactness criterion for \(C_\varphi\) acting on \(\calQ_p\).
Because of the additional supremum over \(a\in\D\) in \eqref{20260606eq09}
below, Theorem~\ref{20260606thm02} does not follow by a direct application of
Theorem~\ref{20260606prop01}. Instead, we shall adapt the proof strategy of
Theorem~\ref{20260606prop01} and carry out the estimates uniformly in
\(a\in\D\).

We turn to some details. For \(0<\rho<1\), \(a\in\D\), and every arc
\(I\subseteq\T\), recall that 
\[
        H_a^{\varphi,\rho}(I)
        :=
        \frac{1}{|I|^{p+2}}
        \int_{Q_I^{\textnormal{up}}\cap\{\rho<|\zeta|<1\}}
        \calN_{p,\varphi,a}(\zeta)\,dA(\zeta).
\]

\begin{proof}[Proof of Theorem~\ref{20260606thm02}]
Assume first that \(C_\varphi\) is bounded on \(\calQ_p\) and that
\begin{equation}\label{20260606eq09}
        \lim_{\rho\to1^-}
        \sup_{a\in\D}
        \sup_{\calD}
        \big\|H_a^{\varphi,\rho}\big\|_{\calB_p(\calD)}
        =0,
\end{equation}
where the second supremum is taken over all dyadic systems \(\calD\) on \(\T\).
We prove that \(C_\varphi\) is compact. By Lemma~\ref{20260606lem06}, it is
enough to consider a bounded sequence \(\{f_n\}\subset \calQ_p\) such that
$f_n \to 0$ uniformly on compact subsets of $\D$, and therefore \(f_n'\to 0\) uniformly on compact subsets of \(\D\).

By the change-of-variables formula for \(\calN_{p,\varphi,a}\),
\begin{equation}\label{20260606eq43}
        \|C_\varphi f_n\|_{\calQ_p,*}^2
        =
        \sup_{a\in\D}
        \int_{\D}|f_n'(\zeta)|^2
        \calN_{p,\varphi,a}(\zeta)\,dA(\zeta).
\end{equation}
Let \(\epsilon>0\). Choose \(\rho<1\) sufficiently close to \(1\) so that
\[
        \sup_{a\in\D}
        \sup_{\calD}
        \big\|H_a^{\varphi,\rho}\big\|_{\calB_p(\calD)}
        <\epsilon.
\]
Applying Theorem \ref{20260606thm03} with
\[
        W(\zeta)
        =
        \one_{\{|\zeta|>\rho\}}
        \calN_{p,\varphi,a}(\zeta),
\]
we obtain
\begin{equation}\label{20260606eq44}
        \int_{\{|\zeta|>\rho\}} |f_n'(\zeta)|^2
        \calN_{p,\varphi,a}(\zeta)\,dA(\zeta)
        \lesssim
        \epsilon\,\sup_n\|f_n\|_{\calQ_p,*}^2 \lesssim \epsilon,
        \qquad a\in\D .
\end{equation}

For the interior part, boundedness of \(C_\varphi\) and the test function
\(f(\zeta)=\zeta\) imply
\begin{equation}\label{20260606eq45}
        \sup_{a\in\D}
        \int_{\D}\calN_{p,\varphi,a}(\zeta)\,dA(\zeta)
        <\infty .
\end{equation}
Therefore,
\[
        \sup_{a\in\D}
        \int_{\{|\zeta|\le\rho\}} |f_n'(\zeta)|^2
        \calN_{p,\varphi,a}(\zeta)\,dA(\zeta)
        \le
        \left(\sup_{|\zeta|\le\rho}|f_n'(\zeta)|^2 \right) \cdot 
        \sup_{a\in\D}
        \int_{\D}\calN_{p,\varphi,a}(\zeta)\,dA(\zeta)
        \to0.
\]
Combining this with \eqref{20260606eq43} and \eqref{20260606eq44}, and then
letting \(\epsilon\to0\), gives
\[
        \|C_\varphi f_n\|_{\calQ_p,*}\to0.
\]
The value at the origin also tends to zero, since
\[
        C_\varphi f_n(0)=f_n(\varphi(0))\to0 \qquad \textrm{as} \quad n \to \infty.
\]
Hence \(\|C_\varphi f_n\|_{\calQ_p}\to0\), and by Lemma~\ref{20260606lem06}, \(C_\varphi\) is compact.

\medskip 

Conversely, assume that \(C_\varphi\) is compact. Then \(C_\varphi\) is
bounded. Suppose that \eqref{20260606eq09} fails. Then there exist
\(\rho_n\to1\), points \(a_n\in\D\), dyadic systems \(\calD_n\), and a constant
\(c_0>0\) such that
\begin{equation}\label{20260606eq46}
        \big\|H_{a_n}^{\varphi,\rho_n}\big\|_{\calB_p(\calD_n)}
        \ge c_0 .
\end{equation}
By Lemma~\ref{20260606lem01}, for each \(n\) there exists a finitely supported
nonnegative sequence \(\{d_I^{(n)}\}_{I\in\calD_n}\) satisfying
\[
        \sup_{J\in\calD_n}
        \frac1{|J|^p}
        \sum_{I\subseteq J}d_I^{(n)}
        \le1
\]
and
\begin{equation}\label{20260606eq47}
        \sum_{I\in\calD_n}
        \frac{d_I^{(n)}}{|I|^{p+2}}
        \int_{Q_I^{\textnormal{up}}\cap\{|\zeta|>\rho_n\}}
        \calN_{p,\varphi,a_n}(\zeta)\,dA(\zeta)
        \ge \frac{c_0}{2}.
\end{equation}
As in the proof of Theorem~\ref{20260606prop01}, we may assume that
\(d_I^{(n)}=0\) unless $Q_I^{\textnormal{up}}\cap\{|\zeta|>\rho_n\}\ne\emptyset$.

Repeating the randomized Bergman-projection realization from the proof of
Theorem~\ref{20260606thm03}, with the weight
\[
        \one_{\{|\zeta|>\rho_n\}}
        \calN_{p,\varphi,a_n}(\zeta),
\]
gives functions \(f_n\in\calQ_p\), normalized by \(f_n(0)=0\), such that
\begin{equation}\label{20260606eq48}
        \|f_n\|_{\calQ_p,*}\lesssim1,
        \qquad
        \int_{\{|\zeta|>\rho_n\}}
        |f_n'(\zeta)|^2
        \calN_{p,\varphi,a_n}(\zeta)\,dA(\zeta)
        \gtrsim1 .
\end{equation}
The constants are independent of \(n\), \(a_n\), the dyadic system \(\calD_n\),
and the weight.

Moreover, the same support argument as in Theorem~\ref{20260606prop01} (see the argument in \eqref{20260606G60} and \eqref{20260606G61})
shows that the functions before applying the Bergman projection are supported in the boundary layer
\[
        \{|\zeta|>1-C_0(1-\rho_n)\}.
\]
Consequently,
\[
        f_n\to0
        \quad\text{locally uniformly on }\D .
\]
Since \(C_\varphi\) is compact, Lemma~\ref{20260606lem06} gives
\[
        \|C_\varphi f_n\|_{\calQ_p}\to0.
\]
On the other hand, by the change-of-variables formula,
\[
        \|C_\varphi f_n\|_{\calQ_p,*}^2
        =
        \sup_{a\in\D}
        \int_{\D}|f_n'(\zeta)|^2
        \calN_{p,\varphi,a}(\zeta)\,dA(\zeta)
        \ge
        \int_{\{|\zeta|>\rho_n\}}
        |f_n'(\zeta)|^2
        \calN_{p,\varphi,a_n}(\zeta)\,dA(\zeta)
        \gtrsim1,
\]
which is a contradiction. Therefore \eqref{20260606eq09} holds, and the proof
is complete.
\end{proof}

\bigskip 
\section{Applications to the $\calQ_p$-Carleson measure problem} \label{Sec06}

As some further applications of our main results, Theorem~\ref{20260606thm03} and its compact counterpart Theorem~\ref{20260606prop01}, we return to the $\calQ_p$-Carleson measure problem discussed in Remark~\ref{20260611rem01}.

As explained there, the full $\calQ_p$-Carleson measure problem is not a direct consequence of the dyadic trace theorem. Nevertheless, for a large class of weights for which a Littlewood--Paley type derivative norm equivalence is available, the problem can be reduced to our current setting. In this section, we use this reduction to obtain boundedness and compactness criteria for the corresponding $\calQ_p$-Carleson measure problems.

We begin with the some definitions.

\begin{defn}
Let $W\geq 0$ be a weight on $\D$. We say that $W$ belongs to the
\emph{Littlewood--Paley class}, and write $W\in\mathcal{LP}$, if for every
$n\in\N$ and every $0<q<\infty$ one has
$$
        \int_{\D}|f(z)|^q W(z)\,dA(z)
        \simeq
        \int_{\D}|f^{(n)}(z)|^q(1-|z|)^{np}W(z)\,dA(z)
        +
        \sum_{j=0}^{n-1}|f^{(j)}(0)|^q
$$
for all $f\in H(\D)$, where the implicit constants may depend on $W,n$ and $q$, but not on $f$.
\end{defn}

The class $\mathcal{LP}$ has been studied extensively in the theory of weighted spaces of analytic
functions. We record below two large classes of examples.

\begin{exa}
\begin{enumerate}
\item [$\bullet$] {\bf Radial weights.} Let $W$ be a radial weight on $\D$, identified with a weight on $[0,1)$, and write 
$$ 
\widehat W(r):=\int_r^1 W(t)\,dt,\qquad 0\le r<1. 
$$ 
We say that 

\vspace{0.1cm}

\begin{enumerate} 
\item[$\bullet$] $W\in\widehat{\mathcal{RD}}$ if there exists $C\geq 1$ such that 
$$ \widehat W(r) \leq C\widehat W\left(\frac{1+r}{2}\right), \qquad 0\le r<1; 
$$ 

\vspace{-0.3cm}
\item[$\bullet$] $W\in\widecheck{\mathcal{RD}}$ if there exist constants $C,K>1$ such that 
$$ 
\widehat W(r) \geq C\widehat W\left(1-\frac{1-r}{K}\right), \qquad 0\le r<1. 
$$ 
\end{enumerate} 
Finally, set 
\vspace{-0.3cm}
$$ 
\mathcal{RD} :=\widehat{\mathcal{RD}}\cap\widecheck{\mathcal{RD}}. 
$$ 
Then,
$$ 
W\in\mathcal{LP} \quad\Longleftrightarrow\quad W\in\mathcal{RD}; 
$$ 
see \cite{PelaezRattya2021,PelaezRosa2026}. Here, as usual, a radial weight
$W$ on $[0,1)$ is identified with the weight on $\D$ given by
$W(z):=W(|z|), \; z\in\D$. 

\medskip 

\item[$\bullet$] \textbf{The generalized Bekoll\'e--Bonami class associated with
$\mathcal{RD}$.}  Let
$\nu\in\mathcal{RD}$ and $1<s<\infty$. A $\nu$--weight $u$ is said to belong
to $B_s(\nu)$ if
$$
\sup_{I\subseteq\T}
\frac{
\left(\displaystyle\int_{Q_I} u(z)\nu(z)dA(z)\right)^{1/s}
\left(\displaystyle\int_{Q_I} u(z)^{-s'/s}\nu(z)dA(z)\right)^{1/s'}
}
{\displaystyle\int_{Q_I}\nu(z)dA(z)}
<\infty,
$$
where $s'$ is the conjugate exponent of $s$. Define
$$
B_\infty(\nu)
:=
\bigcup_{1<s<\infty}
\left\{
W:\frac{W}{\nu}\in B_s(\nu)
\right\},
$$
and
$$
B_\infty(\mathcal{RD})
:=
\bigcup_{\nu\in\mathcal{RD}}B_\infty(\nu).
$$
Then
$$
B_\infty(\mathcal{RD})\subset \mathcal{LP};
$$
see \cite{PelaezRattya2024}. In particular, weights in $B_\infty(\mathcal{RD})$ need not be radial. This
class contains the usual Bekoll\'e--Bonami class $B_\infty$ as a special case,
corresponding to the choice of the reference weight $\nu\equiv 1$. Moreover,
the radial weights in $B_\infty(\mathcal{RD})$ are precisely the weights in
$\mathcal{RD}$.
\end{enumerate}
\end{exa}

As direct consequences of Theorem~\ref{20260606thm03} and its compact counterpart Theorem~\ref{20260606prop01}, we resolve the $\calQ_p$-Carleson measure problem for weights belonging to the Littlewood--Paley class.

For a weight $W$ on $\D$, we write
$$
        \widetilde W(z):=(1-|z|)^2W(z),\qquad z\in\D.
$$

\begin{thm}\label{20260611thm01}
Let $0<p\leq 1$, and let $W\in\mathcal{LP}$. For every arc
$I\subseteq\T$, define
$$
        h_{\widetilde W}(I)
        :=
        \frac{1}{|I|^{p+2}}
        \int_{Q_I^{\textnormal{up}}}\widetilde W(z)\,dA(z).
$$
Then the following statements are equivalent.
\begin{enumerate}
\item[(i)] The identity embedding
$$
        \mathrm{id}:\calQ_p\longrightarrow L^2(W dA)
$$
is bounded; that is, there exists a constant $C>0$ such that
$$
        \int_{\D}|F(z)|^2W(z)\,dA(z)
        \leq C\|F\|_{\calQ_p}^2,
        \qquad F\in\calQ_p.
$$

\item[(ii)] There exists a dyadic system $\calD$ on $\T$ such that
$$
        \|h_{\widetilde W}\|_{\calB_p(\calD)}<\infty,
$$
where, for each fixed $\calD$, the sequence is understood as
$\{h_{\widetilde W}(I)\}_{I\in\calD}$.
\end{enumerate}
Moreover,
$$
        \|\mathrm{id}\|_{\calQ_p\to L^2(W dA)}^2
        \simeq
        W(\D)
        +
        \inf_{\calD}\|h_{\widetilde W}\|_{\calB_p(\calD)}
        \simeq
        W(\D)
        +
        \sup_{\calD}\|h_{\widetilde W}\|_{\calB_p(\calD)},
$$
where the infimum and supremum are taken over all dyadic systems $\calD$
on $\T$.
\end{thm}

\begin{proof}
Since $W\in\mathcal{LP}$, applying the Littlewood--Paley equivalence with
$n=1$ and $q=2$ gives
\begin{equation} \label{20260611eq01}
        \int_{\D}|F(z)|^2W(z)\,dA(z)
        \simeq
        |F(0)|^2
        +
        \int_{\D}|F'(z)|^2\widetilde W(z)\,dA(z),
        \qquad F\in H(\D).
\end{equation}
Thus the boundedness of $\mathrm{id}:\calQ_p\to L^2(W dA)$ is equivalent
to the boundedness of
$$
        \frac{d}{dz}:\calQ_p\longrightarrow L^2(\widetilde W dA),
        \qquad F\mapsto F',
$$
together with the harmless constant term $W(\D)<\infty$, which follows from the assumption that $W\in\mathcal{LP}$. The conclusion therefore follows directly from
Theorem~\ref{20260606thm03}.
\end{proof}

\begin{thm}\label{20260611thm02}
Let $0<p\leq 1$, and let $W\in\mathcal{LP}$. Then the identity embedding
$$
        \mathrm{id}:\calQ_p\longrightarrow L^2(W dA)
$$
is compact if and only if
$$
        \lim_{\rho\to1^-}
        \sup_{\calD}
        \big\|h_{\widetilde W^\rho}\big\|_{\calB_p(\calD)}
        =
        0,
$$
where
$$
        \widetilde W^\rho(z)
        :=
        \widetilde W(z)\mathbf 1_{\{|z|>\rho\}}(z),
$$
and the supremum is taken over all dyadic systems $\calD$ on $\T$.
\end{thm}

\begin{proof}
Again, since $W\in\mathcal{LP}$, \eqref{20260611eq01} holds. Hence compactness of the embedding
$\mathrm{id}:\calQ_p\to L^2(W dA)$ is equivalent to compactness of 
$$
        \frac{d}{dz}:\calQ_p\longrightarrow L^2(\widetilde W dA).
$$
The desired criterion follows from Theorem~\ref{20260606prop01}.
\end{proof}

\end{document}